\documentclass{article}

\title{\huge\textbf{A Novel Approach to Handling the Non-Central Dirichlet Distribution}}
\author{
Carlo Orsi\footnote{e-mail: \mail{orsi.carlo@gmail.com}}
\\
\normalsize Ph.D. in Statistics at University of Milano-Bicocca (Milan, Italy),\\
\normalsize Casella Postale 13, Ufficio Postale Fermo, 63900 Fermo (Italy)
}
\usepackage[utf8]{inputenc}
\usepackage[english]{babel}
\usepackage{geometry}
\usepackage{fancyhdr}
\usepackage{hyperref}
\hypersetup{colorlinks=true, linkcolor=black, citecolor=black, urlcolor=black, plainpages=false}
\usepackage{color}
\usepackage{nameref}
\usepackage{graphicx}
\usepackage{subfigure}
\usepackage{float}
\usepackage{amsthm}
\usepackage{amssymb}
\usepackage{amsmath}
\usepackage{verbatim}
\usepackage{lscape}
\usepackage[font=small,labelfont=bf]{caption}
\usepackage{multirow}
\usepackage{wrapfig}
\newcommand{\mail}[1]{\href{mailto:#1}{\texttt{#1}}}

\usepackage[table]{xcolor}
\usepackage{array}
\newcolumntype{C}[1]{>{\centering\let\newline\\\arraybackslash\hspace{0pt}}m{#1}}
\newcolumntype{L}[1]{>{\raggedright\let\newline\\\arraybackslash\hspace{0pt}}m{#1}}
\newcolumntype{R}[1]{>{\raggedleft\let\newline\\\arraybackslash\hspace{0pt}}m{#1}}
\newtheoremstyle{plain}  
  {\topsep}   
  {\topsep}   
  {\itshape}  
  {}       
  {\bfseries} 
  {}         
  {\newline}  
  {}          
\theoremstyle{plain}
\newtheorem{property}{Property}[section]
\newtheorem{proposition}{\textbf{Proposition}}[section]

\newtheoremstyle{mystyle}  
  {\topsep}   
  {\topsep}   
  {\normalfont}  
  {0pt}       
  {\bfseries} 
  {.}         
  {\newline}  
  {}          
\theoremstyle{mystyle}

\begin{document}
\baselineskip24pt
\maketitle

\begin{abstract}
\baselineskip24pt
In the present paper new insights into the study of the Non-central Dirichlet distribution are provided. This latter is the analogue of the Dirichlet distribution obtained by replacing the Chi-Squared random variables involved in its definition by as many non-central ones. Specifically, a novel approach to tackling the analysis of this model is introduced based on a simple conditional density together with a suitable transposition into the non-central framework of a characterizing property of independent Chi-Squared random variables. This approach thus enables to remedy the undeniable mathematical complexity of the joint density function of such distribution by paving the way towards achieving a new attractive stochastic representation as well as a surprisingly simple closed-form expression for its mixed raw moments. 

\vspace*{0.2cm}
\end{abstract}

\section{Introduction}
\label{sec:introduc}

In the present paper our intent is to provide a valid point of reference for the study of the Non-central Dirichlet distribution \cite{SanNagGup06}, namely the multivariate extension of the Doubly Non-central Beta distribution \cite{OngOrs15}, \cite{Ors21}, this latter being the analogue of the former on the unit interval. In this regard, we are supported by the fact that in recent years such distributions have attracted some interesting applications. For example, in \cite{BotFerBek21} the practical problem of estimating Tsallis entropy in the Bayesian framework when considering a Multinomial likelihood is faced by focusing on previously unconsidered Dirichlet type priors, amongst which the Non-central Dirichlet distribution. In \cite{SchNagWalFla21}, a new non-negative matrix factorization model for data with bounded support between 0 and 1 is introduced based on the Doubly Non-central Beta distribution. Due to the more flexible set of shapes taken on by the density of this latter than that of the Beta and to its simple mixture representation based on a pair of independent Poisson random variables, the resulting model is shown to improve out-of-sample predictive performance on both real and synthetic DNA methylation datasets over state-of-the-art methods in bioinformatics.

That said, our work is organized as follows. Section~\ref{sec:prelim} is devoted to providing the mathematical tool kit and to reviewing the main results of the model under consideration as well as the definition and some relevant properties of the distributions involved in the present study. The core of the matter is addressed in Section~\ref{sec:new.appr.a}, in which a novel approach to analyzing the Non-central Dirichlet distribution is illustrated. This latter is based on an interesting relationship of conditional independence and a suitable conditional density of such a model. The combination of these two ingredients clears the ground for the derivation of a new stochastic representation and of a surprisingly simple solution to the problem of computing the mixed raw moments of the model of interest. In this regard, Section~\ref{sec:simul_res} brings the results of a simulation study aimed at providing numerical validations of the derived moment formula and discussing the advantages of this new formula over the existing one. In conclusion, Section~\ref{sec:concl} contains some final remarks.


\section{Preliminaries}
\label{sec:prelim}

This section is intended to briefly review the probabilistic models involved in the present study as well as to jog the reader's memory about known results of the distribution under consideration in this paper. However, first of all we shall provide the mathematical tool kit aimed at guaranteeing a comprehensive understanding of such distributions. This latter essentially consists of the notions of ascending factorial and generalized hypergeometric function.

Specifically, denoting the gamma function by $\Gamma\left(\cdot\right)$, the $l$-th ascending factorial or Pochhammer's symbol of $a > 0$ is given by
\begin{equation}
\left(a\right)_l=\frac{\Gamma\left(a+l\right)}{\Gamma\left(a\right)}=\left\{\begin{array}{ll} 1 & \mbox{ if } l=0 \\  \\ a\left(a+1\right) \, \ldots \, \left(a+l-1\right) & \mbox{ if } l \in \mathbb{N}\end{array} \right. 
\label{eq:poch.symb}
\end{equation}
\cite{JohKemKot05}, where the non-negative integer $l$ indicates the number of factors appearing in the above product. Some interesting properties of Eq.~(\ref{eq:poch.symb}) are listed in the following. Firstly, by additively decomposing $l$ into $l_1 , \, l_2 \in \mathbb{N}$, one has that
\begin{equation}
\left(a\right)_{l_1+l_2}=\frac{\Gamma\left(a+l_1+l_2\right)}{\Gamma\left(a\right)}=\left\{\begin{array}{l} \frac{\Gamma\left(a+l_1\right)}{\Gamma\left(a\right)} \, \frac{\Gamma\left(a+l_1+l_2\right)}{\Gamma\left(a+l_1\right)}=\left(a\right)_{l_1} \, \left(a+l_1\right)_{l_2} \\ \\ \frac{\Gamma\left(a+l_2\right)}{\Gamma\left(a\right)} \, \frac{\Gamma\left(a+l_2+l_1\right)}{\Gamma\left(a+l_2\right)}=\left(a\right)_{l_2} \, \left(a+l_2\right)_{l_1}\end{array} \right. \, .
\label{eq:poch.symb.sum}
\end{equation}
Then, the ratio of two ascending factorials of $a$ can be expressed as
\begin{equation}
\frac{\left(a\right)_{l_1}}{\left(a\right)_{l_2}}=\left\{\begin{array}{ll} \left(a+l_2\right)_{l_1-l_2} & \mbox{ if } l_1 \geq l_2 \\ \\ \left[\left(a+l_1\right)_{l_2-l_1}\right]^{-1} & \mbox{ if } l_1<l_2\end{array} \right. \, ,
\label{eq:poch.symb.ratio}
\end{equation}
whereas the following expansion holds true for the Pochhammer's symbol of a binomial:
\begin{equation}
\left(a+b\right)_{l}=\sum_{j=0}^{l} {l \choose j} \left(a\right)_{l-j} \left(b\right)_j \, .
\label{eq:poch.symb.binom}
\end{equation}

By virtue of the foregoing concept, the generalized hypergeometric function with $p$ upper parameters and $q$ lower parameters, $p, \, q \in \mathbb{N} \cup \{0\}$, can be accordingly defined as
\begin{equation}
_p^{\, }F_q^{}\left(a_1,\, \ldots \, ,a_p;b_1,\, \ldots \, ,b_q;x\right)=\sum_{i=0}^{+\infty} \frac{(a_1)_i \, \ldots \, (a_p)_i}{(b_1)_i \, \ldots \, (b_q)_i} \frac{x^{i}}{i \, !} \, , \quad x \in \mathbb{R} \, .
\label{eq:fpq}
\end{equation}
For more details on the convergence of the hypergeometric series in Eq.~(\ref{eq:fpq}) as well as for further results and properties of the function $_p^{\, }F_q^{} \, $, the reader is recommended referring to \cite{SriKar85}. The special case of Eq.~(\ref{eq:fpq}) corresponding to $p=q=1$, namely
\begin{equation}
_1^{\, }F_1^{}\left(a; b; x\right)=\sum_{i=0}^{+\infty} \frac{(a)_i}{(b)_i} \frac{x^{i}}{i \, !} \, , \quad x \in \mathbb{R} \, ,
\label{eq:f11}
\end{equation}
is the so-called Kummer's confluent hypergeometric function. This latter plays a prominent role in the analysis of the model of interest. In this regard, a complete list of properties of $_1^{\, }F_1^{}$ is provided for example by \cite{AbrSte64}. In particular, in the sequel a special focus is given to the following well-established recurrence relations holding for contiguous values of the parameters of this function which are recalled herein (see formulas 13.4.5, 13.4.6 in \cite{AbrSte64}):
\begin{eqnarray}
b\left(a+x\right) \, _1^{\, }F_1^{}\left(a; b; x\right)+x\left(a-b\right) \, _1^{\, }F_1^{}\left(a; b+1; x\right)-ab \, _1^{\, }F_1^{}\left(a+1; b; x\right) & = & 0 \, , \nonumber \\
\left(a-1+x\right) \, _1^{\, }F_1^{}\left(a; b; x\right)+\left(b-a\right) \, _1^{\, }F_1^{}\left(a-1; b; x\right)+\left(1-b\right) \, _1^{\, }F_1^{}\left(a; b-1; x\right) & = & 0  \, . \nonumber \\
\label{eq:f11.recurr.relat}
\end{eqnarray}

That said, denoting independence by $\bot$ and letting $D$ be an integer greater than 1, we remember that the $D$-dimensional Dirichlet model with vector of shape parameters $\underline{\alpha}=\left(\alpha_1,\, \ldots \, ,\alpha_{D+1}\right)$, $\alpha_i>0$, $i=1,\, \ldots \,,D+1$, can be defined as follows:
\begin{eqnarray}
\lefteqn{\left\{\begin{array}{l} Y_i \stackrel{\bot}{\sim } \chi^{\, 2}_{ 2 \alpha_i} \quad i=1,\, \ldots \,,D+1 \\ \\ Y^+=\sum_{i=1}^{D+1} Y_i  \end{array} \right. \quad \Rightarrow } \nonumber \\
& \Rightarrow & \quad \underline{X}=\left(X_1,\, \ldots \, ,X_D\right)=\left(\frac{Y_1}{Y^+},\, \ldots \, ,\frac{Y_D}{Y^+}\right) \sim \mbox{Dir}^{\, D}\left(\alpha_1,\, \ldots \, , \alpha_{D+1}\right)
\label{eq:dir.def}
\end{eqnarray}
\cite{KotBalJoh00}. Its joint density function is given by
\begin{equation}
\mbox{Dir}^{\, D}\left(\underline{x};\underline{\alpha}\right)=\frac{\Gamma\left(\alpha^+\right)}{\prod_{i=1}^{D+1}\Gamma\left(\alpha_i\right)}\left[\prod_{i=1}^{D}x_i^{\alpha_i-1}\right] \left(1-\sum_{i=1}^{D}x_i\right)^{\alpha_{D+1}-1} \, , \qquad \underline{x} \in \mathcal{S}^{\, D} \, ,
\label{eq:dir.dens}
\end{equation}
where $\alpha^+=\sum_{i=1}^{D+1}\alpha_i$ and
\begin{equation}
\mathcal{S}^{\, D}=\left\{\underline{x}=\left(x_1,\, \ldots \, ,x_D\right) \, : \; 0 < x_i <1 \, , \; i=1,\, \ldots \, ,D \; , \; \sum_{i=1}^{D} x_i<1\right\}
\label{eq:d.unit.simpl}
\end{equation}
is the unitary simplex in $\mathbb{R}^{\, D}$. For $D=1$, the set specified in Eq.~(\ref{eq:d.unit.simpl}) corresponds to the Real interval $(0,1)$ and the Dirichlet distribution reduces to the Beta one. For the ends of this paper it is useful bearing in mind that, in the notation of Eq.~(\ref{eq:dir.def}), the Dirichlet distribution can be also obtained as conditional distribution of $\underline{X}$ given $Y^+$, the former being independent of the latter by virtue of a characterizing property of independent Gamma random variables \cite{Luk55}. The relevance of this property is remarkable in the present context; therefore, it is made explicit in the following.

\begin{property}[\cite{Luk55} Characterizing property of independent Gamma random variables]
\label{prope:char.prop.chisq}
Let $Y_1$, $Y_2$ be two nondegenerate and positive random variables and suppose that they are independently distributed. The random variables $Y^+=Y_1+Y_2$ and $X_i=Y_i \, / \, Y^+$, for every $i=1,2$, are independently distributed if and only if both $Y_1$ and $Y_2$ has Gamma distributions with the same scale parameter.
\end{property}
\noindent Clearly, Property~\ref{prope:char.prop.chisq} is valid also in the case that the number of the involved independent Gamma random variables is greater than two. In particular, by setting the common scale parameter of these latter equal to $1/2$, such property holds true also for any finite number of independent Chi-Squared random variables. Finally, by Eq.~(\ref{eq:poch.symb}), the mixed raw moment of order $(r_1,r_2)$ of the $\mbox{\normalfont{Dir}}^{\, 2}\left(\alpha_1,\alpha_2,\alpha_3\right)$ distribution can be stated as
\begin{equation}
\mathbb{E}\left(X_1^{\, r_1} \, X_2^{\, r_2}\right)=\frac{\left(\alpha_1\right)_{r_1} \, \left(\alpha_2\right)_{r_2}}{\left(\alpha^+\right)_{r^+}} \, , \qquad \left\{\begin{array}{l} r_1, \, r_2 \in \mathbb{N} \\ \\ r^+=r_1+r_2 \end{array} \right. \, .
\label{eq:dir.mixr1r2mom}
\end{equation}  

The non-central extension of the Chi-Squared model \cite{JohKotBal95} is the main ingredient in the definition and the analysis of the Non-central Dirichlet distribution. Specifically, a Non-central Chi-Squared random variable $Y'$ with $g>0$ degrees of freedom and non-centrality parameter $\lambda \geq 0$, denoted by $\chi'^{\,2}_g \left(\lambda \right)$, can be characterized by means of the following mixture representation:
\begin{equation}
Y' \sim \chi'^{\,2}_g \left(\lambda \right) \qquad \Leftrightarrow \qquad \left\{\begin{array}{l} Y'\,| \, M \; \sim \; \chi^{\, 2}_{g+2M} \\ \\ M \sim \mbox{Poisson}\left(\lambda/2\right) \end{array} \right. \, ,
\label{eq:mixrepres.ncchisq}
\end{equation}
the case $\lambda=0$ corresponding to the $\chi^{\, 2}_g$ distribution. Moreover, such a random variable can be additively decomposed into two independent parts, a central one with $g$ degrees of freedom and a purely non-central one with non-centrality parameter $\lambda$, namely
\begin{equation}
Y'=Y+\sum_{j=1}^{M}F_j \, , \quad \mbox{\normalsize{where}} \; \; \;  Y \sim \chi^{\, 2}_g \quad \bot \quad M \sim \mbox{\normalfont{Poisson}}\left(\lambda/2\right) \quad \bot \quad \{F_j \stackrel{\bot}{\sim } \chi^{\, 2}_2 \} \,  .
\label{eq:sumrepres.ncchisq}
\end{equation}
\noindent By virtue of Eq.~(\ref{eq:sumrepres.ncchisq}), the random variable $Y'_{pnc}=\sum_{j=1}^{M}F_j$ is said to have a Purely Non-central Chi-Squared distribution with non-centrality parameter $\lambda$. Indeed, we shall denote it by $\chi'^{\,2}_0 \left(\lambda \right)$, its number of degrees of freedom being equal to zero \cite{Sie79}. Finally, the Non-central Chi-Squared distribution is reproductive with respect to both the number of degrees of freedom and the non-centrality parameter; specifically:
\begin{equation}
Y'_i \stackrel{\bot}{\sim } \chi'^{\, 2}_{g_i}(\lambda_i) \quad i=1,\, \ldots \, ,m  \qquad \Rightarrow \qquad \left\{\begin{array}{l} Y'^{+}=\sum_{i=1}^m Y'_i \sim \chi'^{\, 2}_{g^+}(\lambda^+) \\ \\ g^+=\sum_{i=1}^m g_i \, , \; \lambda^+=\sum_{i=1}^m \lambda_i \end{array} .\right.
\label{eq:ncchisq.reprod}
\end{equation}

We finally end the present section recalling the definition and a number of properties of the model at study. Specifically, the $D$-dimensional Non-central Dirichlet model with vector of shape parameters $\underline{\alpha}=\left(\alpha_1,\, \ldots \, ,\alpha_{D+1}\right)$ and vector of non-centrality parameters $\underline{\lambda}=\left(\lambda_1,\, \ldots \, ,\lambda_{D+1}\right)$, $\lambda_i \geq 0$, $i=1,\, \ldots \,,D+1$, $D \geq 2$, denoted by $\mbox{NcDir}^{\, D}\left(\underline{\alpha},\underline{\lambda}\right)$, can be easily defined by replacing the $Y_i \,$'s by $Y'_i \stackrel{\bot}{\sim } \chi'^{\,2}_{2 \alpha_i} \left(\lambda_i \right)$ in Eq.~(\ref{eq:dir.def}) as follows:
\begin{eqnarray}
\lefteqn{\left\{\begin{array}{l} Y'_i \stackrel{\bot}{\sim } \chi'^{\,2}_{2 \alpha_i} \left(\lambda_i \right) \quad i=1,\, \ldots \, ,D+1 \\ \\ Y'^{+}=\sum_{i=1}^{D+1} Y'_i  \end{array} \right. \Rightarrow } \nonumber \\
& \Rightarrow & \quad \underline{X}'=\left(X'_1,\, \ldots \, ,X'_D\right)=\left(\frac{Y'_1}{Y'^{+}},\, \ldots \, ,\frac{Y'_D}{Y'^{+}}\right) \sim \mbox{NcDir}^{\, D}\left(\underline{\alpha},\underline{\lambda}\right)
\label{eq:ncdir.def}
\end{eqnarray}
\noindent \cite{SanNagGup06}. The joint density function of the $\mbox{NcDir}^{\, D}\left(\underline{\alpha},\underline{\lambda}\right)$ distribution can be readily derived by following the next arguments; in this regard, let:
\begin{eqnarray}
\lefteqn{\underline{M}=\left(M_1,\, \ldots \, ,M_{D+1}\right) \sim \mbox{\normalfont{Multi-Poisson}}^{\, D+1}(\underline{\lambda} \, / \, 2) \qquad \Leftrightarrow}\nonumber \\
& \qquad \Leftrightarrow \qquad & M_i \stackrel{\bot}{\sim } \mbox{\normalfont{Poisson}}(\lambda_i \, / \, 2) \, , \; \quad  i=1,\, \ldots \, ,D+1 \; .
\label{eq:multipois.def}
\end{eqnarray}
Then, in the notation of Eq.~(\ref{eq:ncdir.def}), by Eq.~(\ref{eq:mixrepres.ncchisq}) the conditional distribution of $Y'_i$ given $\underline{M}$ is of $\chi^{\, 2}_{2 \alpha_i+2 M_i}$ type, $i=1,\, \ldots \, ,D+1$; therefore, by Eq.~(\ref{eq:dir.def}), the following mixture representation holds true for the NcDir distribution:
\begin{equation}
\underline{X}' \sim \mbox{NcDir}^{\, D}\left(\underline{\alpha},\underline{\lambda}\right) \qquad \Leftrightarrow \qquad \left\{\begin{array}{l} \left. \underline{X}' \, \right| \, \underline{M}  \sim \mbox{Dir}^{\, D}\left(\underline{\alpha}+\underline{M}\right) \\ \\ \underline{M} \sim \mbox{\normalfont{Multi-Poisson}}^{\, D+1}(\underline{\lambda} \, / \, 2) \end{array} \right. \, .
\label{eq:ncdir.mixt.repres}
\end{equation}
Hence, the density of $\underline{X}' \sim \mbox{NcDir}^{\, D}\left(\underline{\alpha},\underline{\lambda}\right)$ can be stated as
\begin{equation}
\mbox{NcDir}^{\, D}\left(\underline{x};\underline{\alpha},\underline{\lambda}\right)=\sum_{\underline{j} \in \mathbb{N}_0^{\, D+1}} \left[\Pr\left(\underline{M}=\underline{j}\right) \cdot \mbox{Dir}^{\, D}\left(\underline{x};\underline{\alpha}+\underline{j}\right)\right] \, , \qquad \underline{x} \in \mathcal{S}^{\, D} \, ,
\label{eq:ncdir.dens}
\end{equation}
i.e. as the multiple infinite series of the $\mbox{Dir}^{\, D}\left(\underline{\alpha}+\underline{j}\right)$ densities, $\underline{j} \in \mathbb{N}_0^{\, D+1}$, where $\mathbb{N}_0=\mathbb{N} \cup \{0\}$, weighted by the joint probabilities of the random vector $\underline{M}$ defined in Eq.~(\ref{eq:multipois.def}). The $(D+1)$-variate Multi-Poisson distribution thus plays the role of mixing distribution in the mixture type form of the NcDir$^{\, D}$ density. Moreover, the function in Eq.~(\ref{eq:ncdir.dens}) can be equivalently expressed in terms of perturbation of the corresponding central case in Eq.~(\ref{eq:dir.dens}) as follows:
\begin{eqnarray}
\lefteqn{\mbox{NcDir}^{\, D}\left(\underline{x};\underline{\alpha},\underline{\lambda}\right)=\mbox{Dir}^{\, D}\left(\underline{x};\underline{\alpha}\right) \quad \cdot} \nonumber \\
& \cdot & e^{-\frac{\lambda^+}{2}} \,  \Psi_2^{\, (D+1)}\left[\alpha^+;\underline{\alpha};\frac{\lambda_1}{2}x_1,\, \ldots \, ,\frac{\lambda_D }{2}x_D,\frac{\lambda_{D+1}}{2}\left(1-\sum_{i=1}^{D}x_i\right)\right] \, , \qquad \underline{x} \in \mathcal{S}^{\, D} \qquad
\label{eq:ncdir.perturb.dens}
\end{eqnarray}
where 
\begin{equation}
\Psi_2^{\, (m)}\left[a;b_1,\, \ldots \, ,b_m;x_1,\, \ldots \, ,x_m\right]=\sum_{j_1,\, \ldots \, , \, j_m= \, 0}^{+\infty}\frac{(a)_{j_1+\, \ldots \, +j_m}}{(b_1)_{j_1} \, \ldots \, (b_m)_{j_m}} \, \frac{x_1^{\, j_1}}{j_1!} \, \ldots \, \frac{x_m^{\, j_m}}{j_m!}
\label{eq:ncdir.perturb}
\end{equation}
is the $m$-dimensional ($m>2$) generalization of the Humbert's confluent hypergeometric function $\Psi_2$ \cite{SriKar85}. Unfortunately, the perturbation representation of the Non-central Dirichlet density in Eq.~(\ref{eq:ncdir.perturb.dens}) highlights its poor tractability from a mathematical standpoint on one side and its uneasy interpretability on the other. As a matter of fact, regardless of the constant term, the Dirichlet density is perturbed by a function in $D+1$ variables the only known expression for which, as far as we know, is given by the multiple power series in Eq.~(\ref{eq:ncdir.perturb}); indeed, to our knowledge, this latter function cannot be reduced into an easier form. Therefore, due to its inner structure that is too complex to be handled analytically, the effect of the perturbing factor on the Dirichlet density cannot be clearly identified on varying the parameter vector. 

Finally, amongst the properties of the Non-central Dirichlet model, herein we recall the closure under marginalization, which is shared with the Dirichlet. Specifically, as far as the investigation of the two-dimensional marginals of $\underline{X}'=\left(X'_1, \, \ldots \, ,  X'_D\right) \sim \mbox{\normalfont{NcDir}}^{\, D}\left(\underline{\alpha},\underline{\lambda}\right)$ where $\underline{\alpha}=\left(\alpha_1,\, \ldots \, ,\alpha_{D+1}\right)$ and $\underline{\lambda}=\left(\lambda_1,\, \ldots \, ,\lambda_{D+1}\right)$ is concerned, for every $p,\, q=1, \, \ldots \, ,D$, $p \neq q$, the following result applies:
\begin{equation}
\left(X'_p,X'_q\right) \sim \mbox{\normalfont{NcDir}}^{\, 2}\left(\alpha_p,\alpha_q,\alpha^+-(\alpha_p+\alpha_q),\lambda_p,\lambda_q,\lambda^+-(\lambda_p+\lambda_q)\right) \, . 
\label{eq:ncdir.bidim.marg}
\end{equation}
Clearly, similar results hold for the $k$-dimensional marginals ($k>2$). Last but not the least interesting, the one-dimensional marginals are of Doubly Non-central Beta (DNcB) type; specifically, for every $p=1, \, \ldots \, ,D$:
$$
X'_p \sim \mbox{\normalfont{DNcB}}\left(\alpha_p, \alpha^+-\alpha_p,\lambda_p, \lambda^+-\lambda_p\right) \, .
$$
In this regard, for a first analysis and an in-depth study of the DNcB distribution, the reader can be referred to \cite{OngOrs15} and \cite{Ors21}, respectively.


\section{A new approach of analysis}
\label{sec:new.appr.a}
 
As already stressed, the main weakness of the Non-central Dirichlet distribution lies in its poor tractability from a mathematical standpoint. Despite the growing variety of applications attracted by this model in recent years (see, for example, \cite{BotFerBek21} and \cite{SchNagWalFla21}), the above drawback of the NcDir distribution poses strong limitations on its use as a model for data on the unitary simplex.

That said, in the present section new insights into the study of the model of interest are illustrated. These latter give rise to a novel approach to deeply analyzing it. This approach has its origin in the following realization: in view of the arguments recalled in Section~\ref{sec:prelim}, Property~\ref{prope:char.prop.chisq} is no longer valid in the non-central setting. Just an interesting generalization of this property to the non-central framework is the first ingredient at the root of the above approach. The second one is a suitable conditional density of this distribution expressed as a simple mixture of Dirichlet densities. This latter finding, which is the first to be established herein, enables to overcome the limitations of the NcDir density stemming from its remarkable mathematical complexity. 

\begin{proposition}[Conditional density given $M^+$]
\label{propo:ncdir.cond.distr.mplus}
Let $\underline{X}'=\left(X'_1,\, \ldots \, ,X'_D\right) \sim \mbox{\normalfont{NcDir}}^{\, D}\left(\underline{\alpha},\underline{\lambda}\right)$ where $\underline{\alpha}=\left(\alpha_1,\, \ldots \, ,\alpha_D,\alpha_{D+1}\right)$,  $\underline{\lambda}=\left(\lambda_1,\ldots, \right.$ $\left.\lambda_D,\lambda_{D+1}\right)$ and $\underline{M}=\left(M_1,\, \ldots \, ,M_D,M_{D+1}\right) \sim \mbox{\normalfont{Multi-Poisson}}^{\, D+1}(\underline{\lambda} \, / \, 2)$ with $M^+=\sum_{i=1}^{D+1}M_i$. Then, the conditional density of $\underline{X}'$ given $M^+$ is
\begin{eqnarray}
\lefteqn{f_{\left.\underline{X}' \, \right| \, M^+}\left(\underline{x};\underline{\alpha}, \underline{\lambda}\right)=\sum_{j^+ \, \leq \, M^+}\left[\mbox{\normalfont{Multinomial}}^{\, D}\left(j_1,\, \ldots \, ,j_D;\, M^+,\frac{\lambda_1}{\lambda^+},\, \ldots \, ,\frac{\lambda_D}{\lambda^+}\right) \cdot \right.}\nonumber \\
& \cdot & \left. \mbox{\normalfont{Dir}}^{\, D}\left(\underline{x};\alpha_1+j_1,\, \ldots \, ,\alpha_D+j_D,\alpha_{D+1}+M^+-j^+\right)\right] \qquad \left\{\begin{array}{l} \left(j_1,\, \ldots \, ,j_D\right) \in \mathbb{N}_0^{\, D} \\ \\ j^+=\sum_{i=1}^{D}j_i \end{array} \right.
\label{eq:ncdir.cond.distr.mplus}
\end{eqnarray}
where $\mbox{\normalfont{Multinomial}}^{\, D}\left(j_1,\, \ldots \, ,j_D;\, M^+,\frac{\lambda_1}{\lambda^+},\, \ldots \, ,\frac{\lambda_D}{\lambda^+}\right)$ denotes the joint probability mass function of the $\mbox{\normalfont{Multinomial}}^{\, D}\left(M^+,\frac{\lambda_1}{\lambda^+},\, \ldots \, ,\frac{\lambda_D}{\lambda^+}\right)$ distribution evaluated in $\left(j_1,\, \ldots \, ,j_D\right) \in \mathbb{N}_0^{\, D}$, with $\mathbb{N}_0=\mathbb{N} \cup \{0\}$.
\end{proposition}
\begin{proof}
The proof ensues from the statement 
\begin{eqnarray*}
f_{\left.\underline{X}' \, \right| \, M^+}\left(\underline{x};\underline{\alpha}, \underline{\lambda}\right) & = & \sum_{j^+ \, \leq \, M^+} \left[\, \Pr\left(\left. \, \left(M_1, \, \ldots \, ,M_D\right)=\left(j_1,\, \ldots \, ,j_D\right) \, \right| \, M^+\right) \right. \cdot\\
& \cdot & \left. f_{\left.\underline{X}' \, \right| \, \left(\, M_1, \, \ldots \, ,M_D, M^+\right)}\left(\underline{x};\underline{\alpha}, \underline{\lambda}, \left(j_1, \, \ldots \, ,j_D\right)\right) \, \right] \, ,
\end{eqnarray*}
where $\left(j_1,\, \ldots \, ,j_D\right) \in \mathbb{N}_0^{\, D}$, $j^+=\sum_{i=1}^{D}j_i$, by using the following well-known property of the Multi-Poisson distribution:
\begin{eqnarray}
\lefteqn{\underline{M}=\left(M_1,\, \ldots \, ,M_{D+1}\right) \sim \mbox{\normalfont{Multi-Poisson}}^{\, D+1}(\underline{\lambda} \, / \, 2) \quad \quad \Rightarrow}\nonumber \\
& \Rightarrow & \quad \left. \left(M_1,\, \ldots \, ,M_{D}\right) \, \right| \, M^+ \sim \mbox{\normalfont{Multinomial}}^{\, D}\left(M^+,\frac{\lambda_1}{\lambda^+},\, \ldots \, ,\frac{\lambda_D}{\lambda^+}\right)
\label{eq:multipois.cond.sum}
\end{eqnarray}
and by noting that
\begin{eqnarray}
\lefteqn{\left. \underline{X}' \, \right| \, \left(\, M_1, \, \ldots \, ,M_D,M^+\right) \quad \sim }\nonumber \\
& \sim &  \quad \mbox{\normalfont{Dir}}^{\, D}\left(\alpha_1+M_1,\, \ldots \, ,\alpha_D+M_D,\alpha_{D+1}+M^+-\sum_{i=1}^{D}M_i\right)
\label{eq:eq:ncdir.mixt.repres2}
\end{eqnarray}
is a mere rearrangement of the mixture representation of the $\mbox{\normalfont{NcDir}}^{\, D}$ distribution in Eq.~(\ref{eq:ncdir.mixt.repres}).
\end{proof}

As mentioned before, in the following a Non-central Dirichlet random vector is proved to be independent of the sum of the Non-central Chi-Squared random variables involved in its definition in a suitable conditional form. More precisely, in the notation of Eq.~(\ref{eq:ncdir.def}), this independence relationship applies conditionally on the sum $M^+$ of the components of the Multi-Poisson random vector $\underline{M}$.

\begin{proposition}[Conditional independence given $M^+$]
\label{propo:ncdir.condit.indep}
Let $\underline{X}' \sim \mbox{\normalfont{NcDir}}^{\, D}\left(\underline{\alpha},\underline{\lambda}\right)$ where $\underline{\alpha}=\left(\alpha_1,\, \ldots \, ,\alpha_D,\alpha_{D+1}\right)$, $\underline{\lambda}=\left(\lambda_1, \, \ldots \, ,\lambda_D,\lambda_{D+1}\right)$ and $\underline{M}=\left(M_1,\, \ldots \, ,M_D,M_{D+1}\right) \sim \mbox{\normalfont{Multi-Poisson}}^{\, D+1}(\underline{\lambda} \, / \, 2)$ with $M^+=\sum_{i=1}^{D+1} M_i$. Then, $\underline{X}'$ and $Y'^+$ are conditionally independent given $M^+$.
\end{proposition}
\begin{proof}
In the wake of the statement in Eq.~(\ref{eq:eq:ncdir.mixt.repres2}), by Eq.~(\ref{eq:mixrepres.ncchisq}), the ingredients in the definition of the NcDir$^{\, D}$ model in Eq.~(\ref{eq:ncdir.def}) can be analogously represented as follows:
$$Y'_i \left. \, \right| \, \left(M_1, \, \ldots \, , M_D,M^+\right) \; \stackrel{\bot}{\sim } \; \left\{\begin{array}{ll} \chi^2_{2 \alpha_i+2 M_i} &  , \; i=1, \, \ldots \, , D \\ \\ \chi^2_{2 \alpha_{D+1}+2 \left(M^+-\sum_{r=1}^{D}M_r\right)} & , \; i=D+1 \end{array}\right. \; ;$$
hence, by Eq.~(\ref{eq:ncchisq.reprod}):
\begin{equation}
\left.Y'^+ \, \right| \, \left(M_1, \, \ldots \, , M_D,M^+\right) \stackrel{d}{=} \left. Y'^+ \, \right| \, M^+ \sim \chi^2_{2\alpha^++2M^+} \, ,
\label{eq:ypplus.cond.chisq}
\end{equation}
where $\stackrel{d}{=}$ denotes equality in distribution. By Property~\ref{prope:char.prop.chisq}, $\underline{X}'$ and $Y'^+$ are independent conditionally on $\left(M_1, \, \ldots \, , M_D,M^+\right)$ and therefore, given this latter, the joint distribution of $\left(\underline{X}',Y'^+\right)$ factores into the marginal distributions of $\underline{X}'$ and $Y'^+$. That said, the proof follows by noting that the joint density of $\left. \left(\underline{X}',Y'^+\right) \, \right| \, M^+$ factores into the marginal densities of $\left. \underline{X}' \, \right| \, M^+$ and $\left. Y'^+ \, \right| \, M^+$; indeed, under Eq.~(\ref{eq:ypplus.cond.chisq}):
\begin{eqnarray*}
f_{\left.\left(\underline{X}',\, Y'^+\right)\, \right| \, M^+}\left(\underline{x},y\right) & = & \sum_{j^+ \leq M^+} \left[\, \Pr\left(\left. \, \left(M_1, \, \ldots \, ,M_D\right)=\left(j_1,\, \ldots \, ,j_D\right) \, \right| \, M^+\right) \cdot \right. \\
& \cdot & \left. f_{\left.\left(\underline{X}',\, Y'^+\right) \, \right| \, \left(\, M_1, \, \ldots \, ,M_D,M^+ \, \right)}\left(\underline{x},y; \, \left(j_1, \, \ldots \, ,j_D\right) \, \right) \, \right]=\\
& = & \sum_{j^+ \leq M^+} \left[\, \Pr\left(\left. \, \left(M_1, \, \ldots \, ,M_D\right)=\left(j_1,\, \ldots \, ,j_D\right) \, \right| \, M^+\right)\cdot \right. \\
& \cdot & \left.  f_{\left. \underline{X}' \, \right| \, \left(M_1, \, \ldots \, ,M_D,M^+\right)}\left(\underline{x},\left(j_1, \, \ldots \, ,j_D\right) \, \right) \, \right]  \cdot f_{\left. Y'^+ \, \right| \, \left(M_1, \, \ldots \, ,M_D,M^+\right)}\left(y\right)\\
& = & f_{\left.\underline{X}' \, \right| \, M^+}\left(\underline{x}\right) \cdot f_{\left.Y'^+ \, \right| \, M^+}\left(y\right) \, ,
\end{eqnarray*}
where $\left(j_1,\, \ldots \, ,j_D\right) \in \mathbb{N}_0^{\, D}$, $j^+=\sum_{i=1}^{D}j_i$ and the conditional density $f_{\left. \underline{X}' \, \right| \, M^+}$ of $\underline{X}'$ given $M^+$ is made explicit in Eq.~(\ref{eq:ncdir.cond.distr.mplus}).
\end{proof}

The approach we have just discussed paves the way towards gaining new results about the model under consideration.

Specifically, an attractive stochastic representation of this latter is first established. More precisely, like a Non-central Chi-Squared random variable in Eq.~(\ref{eq:sumrepres.ncchisq}), a Non-central Dirichlet random vector can be represented as a convex linear combination with random weights of a central component and a purely non-central one; a random fraction of that is distributed according to a Dirichlet whereas the remaining part is of a Purely Non-central Dirichlet type.

\begin{proposition}[Representation in terms of stochastic convex linear combination]
\label{propo:rappr.clc.ncdir}
Let $\underline{X}' \sim \mbox{\normalfont{NcDir}}^{\, D}\left(\underline{\alpha},\underline{\lambda}\right)$ and $\underline{M}=\left(M_1,\, \ldots \, ,M_{D+1}\right) \sim \mbox{\normalfont{Multi-Poisson}}^{\, D+1}(\underline{\lambda} \, / \, 2)$ with $M^+=\sum_{i=1}^{D+1}M_i$. Then:
\begin{equation}
\underline{X}'=X'_2 \, \underline{X} +\left(1-X'_2\right) \underline{X}'_{pnc} \; ,
\label{eq:rappr.clc.ncdir}
\end{equation}
\begin{itemize}
\item[i)] where $\underline{X}$, $(X'_2,\underline{X}'_{pnc})$ are mutually independent and $\underline{X} \sim \mbox{\normalfont{Dir}}^{\, D}\left(\underline{\alpha}\right)$, 
\item[ii)] $(X'_2, \underline{X}'_{pnc})$ are conditionally independent given $M^+$, $\left.X'_2 \, \right| \, M^+ \sim \mbox{\normalfont{Beta}}\left(\alpha^+,M^+\right)$ and the density of $\left. \underline{X}'_{pnc} \, \right| \, M^+$ is the special case of Eq.~(\ref{eq:ncdir.cond.distr.mplus}) where $\underline{\alpha}=\underline{0}$.
\end{itemize}
\end{proposition} 
\begin{proof}
In the notation of Eq.~(\ref{eq:ncdir.def}), by Eqs.~(\ref{eq:sumrepres.ncchisq}) and~(\ref{eq:ncchisq.reprod}), one has:
\begin{eqnarray}
\lefteqn{\underline{X}'=\left(\frac{Y'_i}{Y'^+}\right)_{i=1, \, \ldots \, ,D}=\left(\frac{Y_i+\sum_{j=1}^{M_i}F_j}{Y^{+}+\sum_{j=1}^{M^+}F_j}\right)_{i=1, \, \ldots \, ,D}=}\nonumber \\
& = & \left(\frac{Y_i}{Y^++\sum_{j=1}^{M^+}F_j}\right)_{i=1, \, \ldots \, ,D}+\left(\frac{\sum_{j=1}^{M_i}F_j}{Y^++\sum_{j=1}^{M^+}F_j}\right)_{i=1, \, \ldots \, ,D} \, .
\label{eq:rappr.ncdir.1}
\end{eqnarray}
The first term on the right-hand side of Eq.~(\ref{eq:rappr.ncdir.1}) can be restated as
$$
\left(\frac{Y_i}{Y^++\sum_{j=1}^{M^+}F_j}\right)_{i=1, \, \ldots \, ,D}=\frac{Y^+}{Y^++\sum_{j=1}^{M^+}F_j} \cdot \left(\frac{Y_i}{Y^+}\right)_{i=1, \, \ldots \, ,D} \, ;
$$
with respect to the second term, we similarly have:
$$
\left(\frac{\sum_{j=1}^{M_i}F_j}{Y^++\sum_{j=1}^{M^+}F_j}\right)_{i=1, \, \ldots \, ,D}=
\frac{\sum_{j=1}^{M^+}F_j}{Y^++\sum_{j=1}^{M^+}F_j} \cdot \left(\frac{\sum_{j=1}^{M_i}F_j}{\sum_{j=1}^{M^+} F_j}\right)_{i=1, \, \ldots \, ,D} \, ,
$$
which is meaningful provided we set:
$$
\left(\frac{\sum_{j=1}^{M_i}F_j}{\sum_{j=1}^{M^+} F_j}\right)_{i=1, \, \ldots \, ,D}=\underline{0} \qquad \mbox{if }M_i=0, \quad \forall \, i=1, \, \ldots \, , D+1 \; .
$$
Finally, by setting
\begin{equation}
X'_2=\frac{Y^+}{Y^++\sum_{j=1}^{M^+}F_j}, \quad \underline{X}=\left(\frac{Y_i}{Y^+}\right)_{i=1, \, \ldots \, ,D}, \quad \underline{X}'_{pnc}=\left(\frac{\sum_{j=1}^{M_i}F_j}{\sum_{j=1}^{M^+} F_j}\right)_{i=1, \, \ldots \, ,D} \, ,
\label{eq:rappr.ncdir.2}
\end{equation}
the decomposition in Eq.~(\ref{eq:rappr.clc.ncdir}) is established. Now let us consider the random vector $(\underline{X},X'_2,\underline{X}'_{pnc})$. In light of Eq.~(\ref{eq:rappr.ncdir.2}), the marginal random vector $(X'_2,\underline{X}'_{pnc})$ is a function of $(Y^+,\underline{M},\left\{F_j\right\})$; moreover, this latter is independent of $\underline{X}$ as $Y^+$ is independent of $\underline{X}$ by virtue of Property~\ref{prope:char.prop.chisq}. Therefore, $\underline{X}$ and $\left(X'_2,\underline{X}'_{pnc}\right)$ are mutually independent and $\underline{X} \sim \mbox{\normalfont{Dir}}^{\, D}\left(\underline{\alpha}\right)$; hence, result \textit{i)} is proved. In order to prove result \textit{ii)}, observe first that:
\begin{eqnarray}
\left.X'_2 \, \right| \, \underline{M} & = & \left.\frac{Y^+}{Y^++\sum_{j=1}^{M^+}F_j}\right| \, \underline{M} \; \sim \; \mbox{Beta}\left(\alpha^+,M^+\right) \, , \nonumber \\
\left.\underline{X}'_{pnc} \, \right| \, \underline{M} & = & \left. \left(\frac{\sum_{j=1}^{M_i}F_j}{\sum_{j=1}^{M^+} F_j}\right)_{i=1, \, \ldots \, ,D} \right|\,  \underline{M} \; \sim \; \mbox{Dir}^{\, D} \left(\underline{M}\right) \, ;
\label{eq:rappr.ncdir.3}
\end{eqnarray}
moreover, $X'_2$ and $\underline{X}'_{pnc}$ are conditionally independent given $\underline{M}$. That said, the proof of \textit{ii)} follows by noting that the joint density function of $(X'_2,\underline{X}'_{pnc}) \, | 	\, M^+$ factores into the marginal densities of $\left.X'_2 \, \right| \, M^+$ and $\left. \underline{X}'_{pnc} \, \right| \, M^+$, where, in light of Eq.~(\ref{eq:rappr.ncdir.3}), by Eq.~(\ref{eq:ncdir.mixt.repres}), the conditional density $f_{\left. \underline{X}'_{pnc} \, \right| \, M^+}$ of $\underline{X}'_{pnc}$ given $M^+$ corresponds to Eq.~(\ref{eq:ncdir.cond.distr.mplus}) where $\underline{\alpha}$ is set equal to $\underline{0}$. Under Eq.~(\ref{eq:multipois.cond.sum}), one can obtain:
\begin{eqnarray*}
f_{\left.\left(X'_2,\, \underline{X}'_{pnc}\right)\, \right| \, M^+}\left(x_2,\underline{x}\right) & = & \sum_{j^+ \leq M^+} \left[\, \Pr\left(\left. \, \left(M_1, \, \ldots \, ,M_D\right)=\left(j_1,\, \ldots \, ,j_D\right) \, \right| \, M^+\right) \cdot \right. \\
& \cdot & \left. f_{\left.\left(X'_2,\, \underline{X}'_{pnc}\right) \, \right| \, \underline{M}} \left(x_2,\underline{x}\right) \, \right]=\\
& = & \sum_{j^+ \leq M^+} \left[\, \Pr\left(\left. \, \left(M_1, \, \ldots \, ,M_D\right)=\left(j_1,\, \ldots \, ,j_D\right) \, \right| \, M^+\right)\cdot \right. \\
& \cdot & \left.  f_{\left. X'_2 \, \right| \, \underline{M}} \left(x_2\right) \cdot f_{\left. \underline{X}'_{pnc} \, \right| \, \underline{M}}\left(\underline{x}\right) \, \right] =\\
& = & \sum_{j^+ \leq M^+} \left[\, \Pr\left(\left. \, \left(M_1, \, \ldots \, ,M_D\right)=\left(j_1,\, \ldots \, ,j_D\right) \, \right| \, M^+\right)\cdot \right. \\
& \cdot & \left. \mbox{Dir}^{\, D}\left(\underline{x};j_1,\, \ldots \, ,j_D,M^+-j^+\right) \right] \cdot \mbox{Beta}\left(x_2;\alpha^+,M^+\right)=\\
& = & f_{\left. \underline{X}'_{pnc} \, \right| \, M^+}\left(\underline{x}\right) \cdot f_{\left. X'_2 \, \right| \, M^+}\left(x_2\right) \, ,
\end{eqnarray*}
where $\left(j_1,\, \ldots \, ,j_D\right) \in \mathbb{N}_0^{\, D}$ and $j^+=\sum_{i=1}^{D}j_i$.
\end{proof}

Then, a surprisingly simple expression for the mixed raw moments of the Non-central Dirichlet distribution is derived. In this regard, for the sake of simplicity we shall focus on the case $D=2$; however, by Eq.~(\ref{eq:ncdir.bidim.marg}), such result also holds true for any two-dimensional marginal of the NcDir$^{\, D}$ model with $D>2$. 

By analogy with the form of the density in Eq.~(\ref{eq:ncdir.dens}) and by Eq.~(\ref{eq:dir.mixr1r2mom}), for every $r_1, \, r_2 \in \mathbb{N}$, the mixed raw moment of order $(r_1, r_2)$ of $\left(X'_1,X'_2\right) \sim \mbox{\normalfont{NcDir}}^{\, 2}\left(\alpha_1,\alpha_2,\alpha_3,\lambda_1,\lambda_2,\lambda_3\right)$ can be stated as
\begin{eqnarray}
\lefteqn{\mathbb{E}\left[\left(X_1'\right)^{ \, r_1} \left(X_2'\right)^{ \, r_2} \, \right]=  \qquad \qquad \qquad \qquad \qquad \qquad \left\{\begin{array}{l} r^+=r_1+r_2 \\ j^+=j_1+j_2+j_3 \end{array}\right. }\nonumber \\
& = &  \sum_{j_1, \, j_2, \, j_3 \, = \, 0}^{+\infty} \left\{ \, \Pr\left[ \; \left(M_1,M_2,M_3\right)=\left(j_1,j_2,j_3\right) \; \right] \cdot \frac{\left(\alpha_1+j_1\right)_{r_1} \, \left(\alpha_2+j_2\right)_{r_2}}{\left(\alpha^+ +j^+\right)_{r^+}} \, \right\}  \, ,
\label{eq:ncdir.mixr1r2mom.def}
\end{eqnarray}
i.e. as the multiple infinite series of the mixed raw moments of order $(r_1,r_2)$ of the $\mbox{Dir}^{\, 2}\left(\alpha_1+j_1,\alpha_2+j_2,\alpha_3+j_3\right)$ distributions, $j_i \in \mathbb{N} \cup \{0\}$, $i=1,2,3$, weighted by the $\mbox{Multi-Poisson}^{\, 3}\left(\frac{\lambda_1}{2},\frac{\lambda_2}{2},\frac{\lambda_3}{2}\right)$ probabilities. Moreover, by Eq.~(\ref{eq:poch.symb.sum}), Eq.~(\ref{eq:ncdir.mixr1r2mom.def}) can be equivalently expressed as the following doubly infinite sum of generalized hypergeometric functions with two parameters both at numerator and denominator:
\begin{eqnarray}
\lefteqn{\mathbb{E}\left[\left(X_1'\right)^{ \, r_1} \left(X_2'\right)^{ \, r_2} \, \right]=}\nonumber \\
& = & \frac{\left(\alpha_1\right)_{r_1} \left(\alpha_2\right)_{r_2}}{\left(\alpha^+\right)_{r^+}} \, e^{-\frac{\lambda^+}{2}} \, \sum_{j_3 \, = \, 0}^{+\infty} \frac{\left(\alpha^+\right)_{j_3}}{\left(\alpha^++r^+\right)_{j_3}} \frac{\left(\frac{\lambda_3}{2}\right)^{j_3}}{j_3!} \sum_{j_2 \, = \, 0}^{+\infty} \frac{\left(\alpha_2+r_2\right)_{j_2} \left(\alpha^++j_3\right)_{j_2}}{\left(\alpha_2\right)_{j_2} \left(\alpha^++r^++j_3\right)_{j_2}} \frac{\left(\frac{\lambda_2}{2}\right)^{j_2}}{j_2!} \nonumber \\
& \cdot & \, _2^{\, }F^{\, }_2\left(\alpha_1+r_1,\alpha^++j_2+j_3;\alpha_1,\alpha^++r^++j_2+j_3;\frac{\lambda_1}{2}\right) \, .
\label{eq:ncdir.mixr1r2mom.def2}
\end{eqnarray}
Albeit with some difficulties, Eq.~(\ref{eq:ncdir.mixr1r2mom.def2}), being less complicated than Eq.~(\ref{eq:ncdir.mixr1r2mom.def}), can be used to compute the quantity of interest in any statistical package where the function $_p^{\, }F^{\, }_q$ is already implemented. This formula is computationally cumbersome anyway. Hence, we shall provide a new interesting result that allows the computation at issue to be reduced from the doubly infinite sum in Eq.~(\ref{eq:ncdir.mixr1r2mom.def2}) to an easily handeable form given by the following doubly finite sum.

\begin{proposition}[Mixed raw moments of the bivariate distribution]
\label{propo:ncdir.momr1r2}
The mixed raw moment of order $(r_1,r_2)$ of $\left(X'_1,X'_2\right) \sim \mbox{\normalfont{NcDir}}^{\, 2}\left(\alpha_1,\alpha_2,\alpha_3,\lambda_1,\lambda_2,\lambda_3\right)$ admits the following expression:
\begin{eqnarray}
\lefteqn{\mathbb{E}\left[\left(X'_1\right)^{r_1} \left(X'_2\right)^{r_2}\right]= \frac{\left(\alpha_1\right)_{r_1}  \left(\alpha_2\right)_{r_2}}{\left(\alpha^+\right)_{r^+}}  \; e^{-\frac{\lambda^+}{2}} \cdot \qquad \quad \left\{\begin{array}{l} r_1,r_2 \in \mathbb{N}, \, r^+=r_1+r_2 \\ j^+=j_1+j_2 \end{array}\right.} \nonumber\\
& \cdot & \sum_{j_1=0}^{r_1} \sum_{j_2=0}^{r_2} \frac{{r_1 \choose j_1} {r_2 \choose j_2} \left(\alpha^+\right)_{j^+} \left(\frac{\lambda_1}{2}\right)^{j_1} \left(\frac{\lambda_2}{2}\right)^{j_2}}{\left(\alpha^++r^+\right)_{j^+} \left(\alpha_1\right)_{j_1} \left(\alpha_2\right)_{j_2}} \, _1F_1\left(\alpha^++j^+;\alpha^++r^++j^+;\frac{\lambda^+}{2}\right)
\label{eq:ncdir.momr1r2}
\end{eqnarray}
\end{proposition}
\begin{proof}
Let $\left(L_1,L_2\right)$ have conditional distribution given $M^+$ of $\mbox{Multinomial}^{\, 2}\left(M^+,\theta_1,\theta_2\right)$ type with $\theta_i=\lambda_i \, / \, \lambda^+$, $i=1,2$. By Eq.~(\ref{eq:ncdir.cond.distr.mplus}), one has:
\begin{eqnarray}
\lefteqn{\mathbb{E}\left[\left. \left(X'_1\right)^{r_1} \left(X'_2\right)^{r_2} \right| \, M^+\right]=}\nonumber \\
& = & \int_{\left(x_1,\, x_2\right) \, \in \, \mathcal{S}^2} x_1^{\, r_1} x_2^{\, r_2} \cdot f_{\left.\left(X'_1,\, X'_2\right) \, \right| \, M^+}\left(x_1,x_2;\alpha_1,\alpha_2,\alpha_3,\lambda_1,\lambda_2,\lambda_3\right) \, dx_1 \, dx_2=\nonumber \\
& = & \sum_{l_1=0}^{M^+} \sum_{l_2=0}^{M^+- \, l_1} \frac{\left(\alpha_1+l_1\right)_{r_1} \left(\alpha_2+l_2\right)_{r_2} }{\left(\alpha^++M^+\right)_{r^+}} \, {M^+ \choose l_1 \; l_2} \; \theta_1^{\, l_1} \,  \theta_2^{\, l_2} \left(1-\theta_1-\theta_2\right)^{M^+- \, l_1-l_2} = \nonumber \\
& = & \frac{\mathbb{E}\left[\left. \left(\alpha_1+L_1\right)_{r_1} \left(\alpha_2+L_2\right)_{r_2} \right| \, M^+\right]}{\left(\alpha^++M^+\right)_{r^+}} \, ;
\label{eq:mom.dim1}
\end{eqnarray}
in light of Eq.~(\ref{eq:poch.symb.binom}):
$$\left(\alpha_i+L_i\right)_{r_i}=\left[\left(\alpha_i-1\right)+\left(L_i+1\right)\right]_{r_i}=\sum_{j_i=0}^{r_i} {r_i \choose j_i} \left(\alpha_i-1\right)_{r_i-j_i} \left(L_i+1\right)_{j_i} \, , \quad \; i=1,2 \, ,$$
so that one obtains:
\begin{eqnarray}
\lefteqn{\mathbb{E}\left[\left. \left(\alpha_1+L_1\right)_{r_1} \left(\alpha_2+L_2\right)_{r_2} \right| \, M^+\right]=} \nonumber \\
& = & \sum_{j_1=0}^{r_1} \sum_{j_2=0}^{r_2} {r_1 \choose j_1} {r_2 \choose j_2} \left(\alpha_1-1\right)_{r_1-j_1} \left(\alpha_2-1\right)_{r_2-j_2} \mathbb{E}\left[\left. \left(L_1+1\right)_{j_1} \left(L_2+1\right)_{j_2} \right| \, M^+\right] \, , \nonumber \\
\label{eq:mom.dim1b}
\end{eqnarray}
where:
\begin{eqnarray}
\lefteqn{\mathbb{E}\left[\left. \left(L_1+1\right)_{j_1} \left(L_2+1\right)_{j_2} \right| \, M^+\right]=} \nonumber \\
& = & \sum_{l_1=0}^{M^+} \sum_{l_2=0}^{M^+-\, l_1} \left(l_1+1\right)_{j_1} \left(l_2+1\right)_{j_2} \, {M^+ \choose l_1 \; l_2} \; \theta_1^{\, l_1} \,  \theta_2^{\, l_2} \left(1-\theta_1-\theta_2\right)^{M^+- \, l_1- \, l_2} \, .
\label{eq:mom.dimb2}
\end{eqnarray}
By Eq.~(\ref{eq:poch.symb}), for every $j_i=0,\, \ldots \, ,r_i$, $i=1,2$:
\begin{equation}
\left(l_i+1\right)_{j_i}=\frac{\Gamma\left(l_i+j_i+1\right)}{\Gamma\left(l_i+1\right)}=\frac{\left(l_i+j_i\right) !}{l_i !}={l_i+j_i \choose l_i} \, j_i! \; ;
\label{eq:mom.dimb3}
\end{equation}
under Eq.~(\ref{eq:mom.dimb3}), Eq.~(\ref{eq:mom.dimb2}) can be thus rewritten as follows: 
\begin{eqnarray}
\lefteqn{\mathbb{E}\left[\left. \left(L_1+1\right)_{j_1} \left(L_2+1\right)_{j_2} \right| \, M^+\right]=j_1 ! \; j_2 ! \, \sum_{l_1=0}^{M^+} {l_1 +j_1 \choose l_1} {M^+ \choose l_1} \, \theta_1^{\, l_1}  \cdot} \nonumber \\
& \cdot & \sum_{l_2=0}^{M^+-\, l_1} {l_2 +j_2 \choose l_2} {M^+-\, l_1 \choose l_2} \left[\left(1-\theta_1\right)-\theta_2\right]^{\left(M^+-\, l_1\right)- \, l_2} \theta_2^{\, l_2}  \, .
\label{eq:mom.dim1c}
\end{eqnarray}
In carrying out the prove, reference must be made to Ljunggren's Identity, namely
\begin{equation}
\sum_{k=0}^{n} {\alpha+k \choose k} {n \choose k} \left(x-y\right)^{n-k} y^k=\sum_{k=0}^{n} {\alpha \choose k} {n \choose k} \, x^{n-k} \, y^k \, , 
\label{eq:ljunggren.id}
\end{equation}
which is (3.18) in \cite{Gou72}. More precisely, by the special case of Eq.~(\ref{eq:ljunggren.id}) where $k=l_2$, $n=M^+- \, l_1$, $\alpha=j_2$, $x=1-\theta_1$, $y=\theta_2$, the final sum in Eq.~(\ref{eq:mom.dim1c}) can be equivalently expressed as
$$
\sum_{l_2=0}^{M^+-\, l_1} {j_2 \choose l_2} \, {M^+-\, l_1 \choose l_2} \left(1-\theta_1\right)^{M^+-\, l_1-\, l_2} \theta_2^{\, l_2} \, ,$$
so that, for every $j_i=0,\, \ldots \, ,r_i$, $i=1,2$, Eq.~(\ref{eq:mom.dim1c}) can be restated as
\begin{eqnarray}
\lefteqn{\mathbb{E}\left[\left. \left(L_1+1\right)_{j_1} \left(L_2+1\right)_{j_2} \right| \, M^+\right]=} \nonumber \\
& = & j_1 ! \, j_2 ! \, \sum_{l_1=0}^{M^+} \sum_{l_2=0}^{M^+-\, l_1} {l_1+j_1 \choose l_1}  {j_2 \choose l_2} {M^+ \choose l_1 \; l_2}  \,  \left(1-\theta_1\right)^{M^+- \, l_1-\, l_2} \theta_1^{\, l_1} \,  \theta_2^{\, l_2} = \nonumber \\
& = & j_1 ! \, j_2 ! \, \sum_{l_2=0}^{M^+} {j_2 \choose l_2} \, {M^+ \choose l_2} \,  \theta_2^{\, l_2} \cdot \nonumber \\
& \cdot & \sum_{l_1=0}^{M^+-\, l_2}  {l_1+j_1 \choose l_1} \, {M^+-l_2 \choose l_1}  \left(1-\theta_1\right)^{(M^+- \, l_2)-\, l_1} \,  \theta_1^{\, l_1}   \, .
\label{eq:mom.dim2a}
\end{eqnarray}
Then, by the special case of Eq.~(\ref{eq:ljunggren.id}) where $k=l_1$, $n=M^+- \, l_2$, $\alpha=j_1$, $x=1$, $y=\theta_1$, the final sum in Eq.~(\ref{eq:mom.dim2a}) can be equivalently expressed as
$$
\sum_{l_1=0}^{M^+-\, l_2} {j_1 \choose l_1} \, {M^+-\, l_2 \choose l_1} \, \theta_1^{\, l_1}$$
and Eq.~(\ref{eq:mom.dim2a}) can be accordingly rewritten in the form of
\begin{eqnarray}
\lefteqn{\mathbb{E}\left[\left. \left(L_1+1\right)_{j_1} \left(L_2+1\right)_{j_2} \right| \, M^+\right]=} \nonumber \\
& = & j_1 ! \, j_2 ! \, \sum_{l_2=0}^{M^+} {j_2 \choose l_2} {M^+ \choose l_2} \,  \theta_2^{\, l_2} \sum_{l_1=0}^{M^+-\, l_2} {j_1 \choose l_1} {M^+-\, l_2 \choose l_1}   \,  \theta_1^{\, l_1} \, ,
\label{eq:mom.dim2b}
\end{eqnarray}
so that, by Eqs.~(\ref{eq:mom.dim1}),~(\ref{eq:mom.dim1b}) and~(\ref{eq:mom.dim2b}), conditionally on $M^+$, the mixed raw moment of interest takes on the following expression:
\begin{eqnarray}
\lefteqn{\mathbb{E}\left[\left. \left(X'_1\right)^{r_1} \left(X'_2\right)^{r_2} \right| \, M^+ \right]=}\nonumber \\ 
& = & \frac{1}{\left(\alpha^++M^+\right)_{r^+}} \sum_{j_1=0}^{r_1} \sum_{j_2=0}^{r_2} {r_1 \choose j_1} j_1 ! \, {r_2 \choose j_2} j_2 ! \left(\alpha_1-1\right)_{r_1-j_1} \left(\alpha_2-1\right)_{r_2-j_2} \cdot \nonumber \\
& \cdot & \sum_{l_2=0}^{M^+} {j_2 \choose l_2} {M^+ \choose l_2} \,  \theta_2^{\, l_2} \sum_{l_1=0}^{M^+-\, l_2} {j_1 \choose l_1} {M^+-\, l_2 \choose l_1}   \,  \theta_1^{\, l_1} \, ;
\label{eq:mom.dim2c}
\end{eqnarray}
applying the law of iterated expectations to Eq.~(\ref{eq:mom.dim2c}) finally leads to:
\begin{eqnarray}
\lefteqn{\mathbb{E}\left[\left(X'_1\right)^{r_1} \left(X'_2\right)^{r_2}\right]=}\nonumber \\ 
& = & e^{-\frac{\lambda^+}{2}} \, \sum_{j_1=0}^{r_1} \sum_{j_2=0}^{r_2} {r_1 \choose j_1} j_1 ! \, {r_2 \choose j_2} j_2 ! \left(\alpha_1-1\right)_{r_1-j_1} \left(\alpha_2-1\right)_{r_2-j_2} \cdot \nonumber \\
& \cdot & \sum_{m=0}^{+\infty} \frac{\left(\frac{\lambda^+}{2}\right)^m \Gamma\left(\alpha^++m\right)}{m ! \, \Gamma\left(\alpha^++r^++m\right)} \sum_{l_2=0}^{m} {j_2 \choose l_2} {m \choose l_2} \,  \theta_2^{\, l_2} \sum_{l_1=0}^{m-\, l_2} {j_1 \choose l_1} {m-\, l_2 \choose l_1}   \,  \theta_1^{\, l_1} \, .
\label{eq:mom.dim3}
\end{eqnarray}
Since ${j_i \choose l_i}=0$ for $l_i>j_i$, $i=1,2$, one has:
$$
l_2 \leq j_2 \leq m \; \Rightarrow \; \left\{\begin{array}{l} l_2=0, \, \ldots \, ,j_2 \\ \\ m=l_2, \, \ldots \, ,+\infty \end{array} \right. \, \quad 
l_1 \leq j_1 \leq m-l_2 \; \Rightarrow \; \left\{\begin{array}{l} l_1=0, \, \ldots \, ,j_1 \\ \\ m=l_1+l_2, \, \ldots \, ,+\infty \end{array} \right.
$$
so that Eq.~(\ref{eq:mom.dim3}) is equivalent to:
\begin{eqnarray*}
\lefteqn{\mathbb{E}\left[\left(X'_1\right)^{r_1} \left(X'_2\right)^{r_2}\right]=}\\ 
& = & e^{-\frac{\lambda^+}{2}} \, \sum_{j_1=0}^{r_1} {r_1 \choose j_1} j_1 ! \left(\alpha_1-1\right)_{r_1-j_1} \, \sum_{l_1=0}^{j_1} \frac{\theta_1^{\, l_1}}{l_1 !} \, {j_1 \choose l_1} \, \sum_{j_2=0}^{r_2} {r_2 \choose j_2} j_2 ! \left(\alpha_2-1\right)_{r_2-j_2} \cdot \\
& \cdot & \sum_{l_2=0}^{j_2} \frac{\theta_2^{\, l_2}}{l_2 !} \, {j_2 \choose l_2} \sum_{m=l_1+l_2}^{+\infty} \frac{\left(\frac{\lambda^+}{2}\right)^m  \Gamma\left(\alpha^++m\right)}{\left(m-\, l_1-\, l_2\right) ! \; \Gamma\left(\alpha^++r^++m\right)} \, ;
\end{eqnarray*}
then, by Eqs.~(\ref{eq:poch.symb}) and~(\ref{eq:f11}), setting $k=m-(l_1+l_2) \Leftrightarrow m=k+(l_1+l_2)$ yields:
\begin{eqnarray*}  
& = & e^{-\frac{\lambda^+}{2}} \, \sum_{j_1=0}^{r_1} {r_1 \choose j_1} j_1 ! \left(\alpha_1-1\right)_{r_1-j_1} \, \sum_{l_1=0}^{j_1} \frac{\left(\frac{\lambda_1}{2}\right)^{l_1} {j_1 \choose l_1}}{l_1 !}  \, \sum_{j_2=0}^{r_2} {r_2 \choose j_2} j_2 ! \left(\alpha_2-1\right)_{r_2-j_2} \cdot \\
& \cdot & \sum_{l_2=0}^{j_2} \frac{\left(\frac{\lambda_2}{2}\right)^{l_2} \, {j_2 \choose l_2} \,\Gamma\left(\alpha^++l_1+l_2\right)}{l_2 ! \; \Gamma\left(\alpha^++r^++l_1+l_2\right)} \,  \sum_{k=0}^{+\infty} \frac{\left(\frac{\lambda^+}{2}\right)^k  \left(\alpha^++l_1+l_2\right)_k}{k ! \, \left(\alpha^++r^++l_1+l_2\right)_k}= \\
& = & e^{-\frac{\lambda^+}{2}} \, \sum_{j_1=0}^{r_1} {r_1 \choose j_1} j_1 ! \left(\alpha_1-1\right)_{r_1-j_1} \, \sum_{l_1=0}^{j_1} \frac{\left(\frac{\lambda_1}{2}\right)^{l_1} {j_1 \choose l_1}}{l_1 !}  \, \sum_{j_2=0}^{r_2} {r_2 \choose j_2} j_2 ! \left(\alpha_2-1\right)_{r_2-j_2} \cdot \\
& \cdot & \sum_{l_2=0}^{j_2} \frac{\left(\frac{\lambda_2}{2}\right)^{l_2} \, {j_2 \choose l_2} \, \; \Gamma\left(\alpha^++l_1+l_2\right)}{l_2 ! \; \Gamma\left(\alpha^++r^++l_1+l_2\right)} \,  _1F_1\left(\alpha^++l_1+l_2;\alpha^++r^++l_1+l_2;\frac{\lambda^+}{2}\right)
\end{eqnarray*}
and, by observing that:
$$
\left\{\begin{array}{l} j_i=0, \, \ldots \, , r_i \\ l_i=0, \, \ldots \, , j_i \end{array} \right. \quad\Rightarrow \quad 0 \, \leq \, l_i \, \leq \, j_i \, \leq \, r_i \quad \Rightarrow \quad 	\left\{\begin{array}{l} l_i=0, \, \ldots \, ,r_i \\ j_i=l_i, \, \ldots \, ,r_i\end{array} \right. \, ,
$$
one can obtain: 
\begin{eqnarray}
\lefteqn{\mathbb{E}\left[\left(X'_1\right)^{r_1} \left(X'_2\right)^{r_2}\right]=e^{-\frac{\lambda^+}{2}} \, \sum_{l_1=0}^{r_1} \frac{\left(\frac{\lambda_1}{2}\right)^{l_1}}{\left(l_1 !\right)^2} \cdot} \nonumber \\
& \cdot & \sum_{l_2=0}^{r_2} \frac{\left(\frac{\lambda_2}{2}\right)^{l_2} \, \Gamma\left(\alpha^++l_1+l_2\right) \, _1F_1\left(\alpha^++l_1+l_2;\alpha^++r^++l_1+l_2;\frac{\lambda^+}{2}\right)}{\left(l_2 !\right)^2 \, \Gamma\left(\alpha^++r^++l_1+l_2\right)} \cdot \nonumber \\
& \cdot & \sum_{j_1=\, l_1}^{r_1} {r_1 \choose j_1} \left(\alpha_1-1\right)_{r_1-j_1} \frac{\left(j_1 !\right)^2}{\left(j_1-l_1\right)!}  \, \sum_{j_2=\, l_2}^{r_2} {r_2 \choose j_2} \left(\alpha_2-1\right)_{r_2-j_2} \frac{\left(j_2 !\right)^2}{\left(j_2-l_2\right)!} \, .
\label{eq:mom.dim4}
\end{eqnarray}
By setting $p_i=j_i-l_i \Leftrightarrow j_i=p_i+l_i$, $i=1,2$ and by Eq.~(\ref{eq:poch.symb.binom}), each of the two final sums in Eq.~(\ref{eq:mom.dim4}) turns out to be tantamount to:
\begin{eqnarray}
\lefteqn{\sum_{j_i=l_i}^{r_i} {r_i \choose j_i} \left(\alpha_i-1\right)_{r_i-j_i} \frac{\left(j_i !\right)^2}{\left(j_i-l_i\right)!}=\sum_{p_i=0}^{r_i-l_i} {r_i \choose p_i+l_i} \left(\alpha_i-1\right)_{r_i-p_i-l_i} \frac{\left[\left(p_i+l_i\right)!\right]^2}{p_i!}}\nonumber \\
& = & \frac{r_i !}{\left(r_i-l_i\right)!}\sum_{p_i=0}^{r_i-l_i} {r_i-l_i \choose p_i} \left(\alpha_i-1\right)_{r_i-p_i-l_i} \, \left(l_i+p_i\right)!= \nonumber \\
& = & \frac{r_i ! \, l_i !}{\left(r_i-l_i\right)!}\sum_{p_i=0}^{r_i-l_i} {r_i-l_i \choose p_i} \left(\alpha_i-1\right)_{r_i-p_i-l_i} \, \left(l_i+1\right)_{p_i}=\nonumber \\
& = & \frac{r_i ! \, l_i !}{\left(r_i-l_i\right)!}\left[\left(\alpha_i-1\right)+\left(l_i+1\right)\right]_{r_i-l_i}=\frac{r_i ! \, l_i !}{\left(r_i-l_i\right)!} \left(\alpha_i+l_i\right)_{r_i-l_i} \, ;
\label{eq:mom.dim5}
\end{eqnarray}
hence, by Eq.~(\ref{eq:mom.dim5}), Eq.~(\ref{eq:mom.dim4}) can be stated in the following form:
\begin{eqnarray}
\lefteqn{\mathbb{E}\left[\left(X'_1\right)^{r_1} \left(X'_2\right)^{r_2}\right]=e^{-\frac{\lambda^+}{2}} \, \sum_{l_1=0}^{r_1} \sum_{l_2=0}^{r_2}  } \nonumber \\
&  & \left[ \frac{{r_1 \choose l_1} \, {r_2 \choose l_2} \left(\frac{\lambda_1}{2}\right)^{l_1} \left(\frac{\lambda_2}{2}\right)^{l_2} \, \Gamma\left(\alpha^++l_1+l_2\right) \, \left(\alpha_1+l_1\right)_{r_1-l_1} \, \left(\alpha_2+l_2\right)_{r_2-l_2}}{\Gamma\left(\alpha^++r^++l_1+l_2\right)} \cdot  \right. \nonumber \\
& \cdot & \left. _1F_1\left(\alpha^++l_1+l_2;\alpha^++r^++l_1+l_2;\frac{\lambda^+}{2}\right) \right] \, .
\label{eq:mom.dim6}
\end{eqnarray}
Eq.~(\ref{eq:mom.dim6}) can be finally exhibited in the same form as in Eq.~(\ref{eq:ncdir.momr1r2}) by noting that:
\begin{eqnarray*}
\Gamma\left(\alpha^++l_1+l_2\right) & = & \Gamma\left(\alpha^+\right) \, \left(\alpha^+\right)_{l_1+l_2} \,  ,\\
\Gamma\left(\alpha^++r^++l_1+l_2\right) & = & \Gamma\left(\alpha^+\right) \, \left(\alpha^+\right)_{r^+} \, \left(\alpha^++r^+\right)_{l_1+l_2} \,  ,\\
\left(\alpha_i+l_i\right)_{r_i-l_i} & = & \frac{\left(\alpha_i\right)_{r_i}}{\left(\alpha_i\right)_{l_i}} \, \qquad i=1,2
\end{eqnarray*}
in light of Eqs.~(\ref{eq:poch.symb}),~(\ref{eq:poch.symb.sum}) and~(\ref{eq:poch.symb.ratio}).
\end{proof}

Hence, the formula of the mixed raw moment of order $(1,1)$ of the $\mbox{\normalfont{NcDir}}^{\, 2}$ distribution can be obtained by taking $r_1=r_2=1$ in Eq.~(\ref{eq:ncdir.momr1r2}) as follows:
\begin{eqnarray}
\mathbb{E}\left(X'_1 \, X'_2\right) & = & \frac{\alpha_1 \, \alpha_2}{\left(\alpha^+\right)_2} \; e^{-\frac{\lambda^+}{2}} \, _1F_1\left(\alpha^+;\alpha^++2;\frac{\lambda^+}{2}\right)+ \nonumber \\
& + & \frac{\alpha_1 \, \frac{\lambda_2}{2}+ \alpha_2 \, \frac{\lambda_1}{2}}{\left(\alpha^++1\right)_2} \; e^{-\frac{\lambda^+}{2}} \, _1F_1\left(\alpha^++1;\alpha^++3;\frac{\lambda^+}{2}\right)+ \nonumber \\
& + & \frac{\frac{\lambda_1}{2} \, \frac{\lambda_2}{2}}{\left(\alpha^++2\right)_2} \; e^{-\frac{\lambda^+}{2}} \,  _1F_1\left(\alpha^++2;\alpha^++4;\frac{\lambda^+}{2}\right) \, .
\label{eq:ncdir.mom11}
\end{eqnarray}
This latter can be algebraically manipulated with the aim to reduce the number of distinct functions $\, _1F_1$ appearing in it. More precisely, a restatement of Eq.~(\ref{eq:ncdir.mom11}) depending only on two functions $_1F_1$ instead of three is derived herein. In this regard, in carrying out the needed computations, reference must be made to the previously specified recurrence identities valid for the function $\, _1F_1$. Specifically, by taking $a=\alpha^++2$, $b=\alpha^++4$, $x=\lambda^+/ \, 2$ in the second formula in Eq.~(\ref{eq:f11.recurr.relat}), one obtains:
\begin{eqnarray}
\lefteqn{2 \, _1F_1\left(\alpha^++1;\alpha^++4;\frac{\lambda^+}{2}\right)=\left(\alpha^++3\right) \cdot} \nonumber \\
& \cdot &   _1F_1\left(\alpha^++2;\alpha^++3;\frac{\lambda^+}{2}\right) -\left(\alpha^++1+\frac{\lambda^+}{2}\right) \, _1F_1\left(\alpha^++2;\alpha^++4;\frac{\lambda^+}{2}\right) \, , \quad \quad
\label{eq:ncdir.mom11.dim1}
\end{eqnarray}
whereas the special case of the first formula in Eq.~(\ref{eq:f11.recurr.relat}) where $a=\alpha^++1$, $b=\alpha^++3$, $x=\lambda^+/ \, 2$ leads to:
\begin{eqnarray}
\lefteqn{\left(\alpha^++3\right)\left(\alpha^++1+\frac{\lambda^+}{2}\right) \, _1F_1\left(\alpha^++1;\alpha^++3;\frac{\lambda^+}{2}\right)=\left(\alpha^++1\right) \cdot} \nonumber \\
& \cdot & \left(\alpha^++3\right) \, _1F_1\left(\alpha^++2;\alpha^++3;\frac{\lambda^+}{2}\right) +2 \cdot \frac{\lambda^+}{2} \,  _1F_1\left(\alpha^++1;\alpha^++4;\frac{\lambda^+}{2}\right) \, ; \quad
\label{eq:ncdir.mom11.dim2}
\end{eqnarray}
under Eq.~(\ref{eq:ncdir.mom11.dim1}), Eq.~(\ref{eq:ncdir.mom11.dim2}) reduces to:
\begin{eqnarray}
\lefteqn{\left(\alpha^++1+\frac{\lambda^+}{2}\right)\left\{\left(\alpha^++3\right) \left[\, _1F_1\left(\alpha^++1;\alpha^++3;\frac{\lambda^+}{2}\right) \right. \right. +}\nonumber \\
& - & \left. \left. \, _1F_1\left(\alpha^++2;\alpha^++3;\frac{\lambda^+}{2}\right)\right]+\frac{\lambda^+}{2} \, _1F_1\left(\alpha^++2;\alpha^++4;\frac{\lambda^+}{2}\right)\right\}=0
\label{eq:ncdir.mom11.dim3}
\end{eqnarray}
and setting $a=\alpha^++2$, $b=\alpha^++3$, $x=\lambda^+/ \, 2$ in the second formula in Eq.~(\ref{eq:f11.recurr.relat}) yields:
\begin{eqnarray}
\lefteqn{\left(\alpha^++1+\frac{\lambda^+}{2}\right) \,  _1F_1\left(\alpha^++2;\alpha^++3;\frac{\lambda^+}{2}\right)=}\nonumber \\
& = & \left(\alpha^++2\right) \, e^{\frac{\lambda^+}{2}} - \, _1F_1\left(\alpha^++1;\alpha^++3;\frac{\lambda^+}{2}\right) \, ,
\label{eq:ncdir.mom11.dim4}
\end{eqnarray}
so that, under Eq.~(\ref{eq:ncdir.mom11.dim4}), Eq.~(\ref{eq:ncdir.mom11.dim3}) specializes into:
\begin{eqnarray}
\lefteqn{_1F_1\left(\alpha^++2;\alpha^++4;\frac{\lambda^+}{2}\right)=\frac{\alpha^++3}{\frac{\lambda^+}{2}\left(\alpha^++1+\frac{\lambda^+}{2}\right)} \cdot}\nonumber \\
& \cdot & \left[\left(\alpha^++2\right) \, e^{\frac{\lambda^+}{2}}-\left(\alpha^++2+\frac{\lambda^+}{2}\right) \,  _1F_1\left(\alpha^++1;\alpha^++3;\frac{\lambda^+}{2}\right)\right] \, .
\label{eq:ncdir.mom11.dim5} 
\end{eqnarray}
Finally, rearranging Eq.~(\ref{eq:ncdir.mom11}) under Eq.~(\ref{eq:ncdir.mom11.dim5}) leads by simple computations to the following improved expression for the moment under consideration:
\begin{eqnarray*}
\lefteqn{\mathbb{E}\left(X'_1 \, X'_2\right)=\frac{\alpha_1 \, \alpha_2}{\left(\alpha^+\right)_2} \, e^{-\frac{\lambda^+}{2}} \, _1F_1\left(\alpha^+;\alpha^++2;\frac{\lambda^+}{2}\right)+} \\
& + & \frac{1}{\alpha^++2} \left[\frac{\alpha_1 \frac{\lambda_2}{2}+\alpha_2 \frac{\lambda_1}{2}}{\alpha^++1}-\frac{\frac{\lambda_1}{2} \frac{\lambda_2}{2}}{\alpha^++1+\frac{\lambda^+}{2}} \right] \, e^{-\frac{\lambda^+}{2}} \, _1F_1\left(\alpha^++1;\alpha^++3;\frac{\lambda^+}{2}\right)+ \\
& + & \frac{\frac{\lambda_1}{2} \frac{\lambda_2}{2}}{\frac{\lambda^+}{2}\left(\alpha^++1+\frac{\lambda^+}{2}\right)} \left[1-e^{-\frac{\lambda^+}{2}} \, _1F_1\left(\alpha^++1;\alpha^++3;\frac{\lambda^+}{2}\right)\right] \, . \qquad \qquad
\end{eqnarray*}


\section{Simulation results}
\label{sec:simul_res}

A simulation study aimed at confirming the validity of the proposed approach is clearly needed. Its accomplishment naturally requires the generation from the model at study. In this regard, the issue of simulating from the Non-central Dirichlet distribution is addressed herein by making use of the algorithm based on its definition in Eq.~(\ref{eq:ncdir.def}). Specifically, this latter simply demands to generate from independent Non-central Chi-Squared random variables. 

That said, numerical validations are first produced for the formula of the mixed raw moments in Eq.~(\ref{eq:ncdir.momr1r2}). Specifically, the following procedure is employed for this purpose. In detail, a sample of $n=30$ series of $N=10000$ random draws is generated from the model of interest for selected values of the parameter vector. The descriptive mixed raw moment of order $(r_1,r_2) \in \left\{(1,1), (1,2), (2,1), (2,2) \right\}$, given by $m_{(r_1,r_2)}\left(X'_1,X'_2\right)=\frac{1}{N}\sum_{i=1}^{N}x^{r_1}_{i1} \, x^{r_2}_{i2}$, is thus computed for each series; then, the sample mean and the standard deviation of this latter quantity are assessed for every above specified value of $(r_1,r_2)$. Hence, the two-tailed $Z$ test for large samples is used to check the null hypothesis that the true mean is equal to the value of the mixed raw moment of order $(r_1,r_2)$ of the considered model, $\mu_{(r_1,r_2)}\left(X'_1,X'_2\right)=\mathbb{E}[(X'_1)^{r_1} (X'_2)^{r_2}]$. The obtained results are listed in Table~\ref{tab:moment.test.ncdir} from which it is noticeable that all the conclusions are in favour of the non-rejection of the null hypothesis at study.

\begin{table}[ht]
\caption{Sample means ($\bar{x}$) and standard deviations ($s$) of the descriptive mixed raw moments of order $(r_1,r_2) \in \left\{(1,1), (1,2), (2,1), (2,2) \right\}$ of $n=30$ series of 10000 random draws from $\left(X'_1,X'_2\right) \sim \mbox{\normalfont{NcDir}}^{\, 2}\left(\alpha_1,\alpha_2,\alpha_3,\lambda_1,\lambda_2,\lambda_3\right)$ for selected values of the parameter vector and $p \, $-values of the two-tailed $Z$ test to check the null hypothesis that the true mean is equal to the mixed raw moment $\mu_{(r_1,r_2)}\left(X'_1,X'_2\right)$ of order $(r_1,r_2)$ of the model at study.}

\vspace{0.25cm}
\centering
\begin{tabular}{c||c|c|c|c|c}
$\left(\alpha_1,\alpha_2,\alpha_3,\lambda_1,\lambda_2,\lambda_3\right)$ & $\left(r_1,r_2\right)$ & $\mu_{(r_1,r_2)}\left(X'_1,X'_2\right)$ & $\bar{x}$ & $s$ & $p$-value\\ \hline \hline
\multirow{4}{3.65cm}{$(0.5,0.6,0.4,1.7,6.4,3.8)$} & (1,1) & 0.07426 & 0.07428 & 0.00079 & 0.91441\\
\cline{2-6} & \multicolumn{1}{c| }{(1,2)} & \multicolumn{1}{c| }{0.03819} & \multicolumn{1}{c| }{0.03820} & \multicolumn{1}{c| }{0.00048} & \multicolumn{1}{c}{0.88320}\\ 
\cline{2-6} & \multicolumn{1}{c| }{(2,1)} & \multicolumn{1}{c| }{0.02190} & \multicolumn{1}{c| }{0.02190} & \multicolumn{1}{c| }{0.00034} & \multicolumn{1}{c}{0.98675} \\ 
\cline{2-6} & \multicolumn{1}{c| }{(2,2)} & \multicolumn{1}{c| }{0.00973} & \multicolumn{1}{c| }{0.00974} & \multicolumn{1}{c| }{0.00017} & \multicolumn{1}{c}{0.76738}\\ \hline \hline
\multirow{4}{3.65cm}{$(0.2,0.3,1.6,1.3,5.5,4.2)$} & (1,1) & 0.03330 & 0.03331 & 0.00042 & 0.89517\\
\cline{2-6} & \multicolumn{1}{c| }{(1,2)} & \multicolumn{1}{c| }{0.01432} & \multicolumn{1}{c| }{0.01431} & \multicolumn{1}{c| }{0.00023} & \multicolumn{1}{c}{0.75554}\\ 
\cline{2-6} & \multicolumn{1}{c| }{(2,1)} & \multicolumn{1}{c| }{0.00821} & \multicolumn{1}{c| }{0.00821} & \multicolumn{1}{c| }{0.00013} & \multicolumn{1}{c}{0.97837} \\ 
\cline{2-6} & \multicolumn{1}{c| }{(2,2)} & \multicolumn{1}{c| }{0.00308} & \multicolumn{1}{c| }{0.00307} & \multicolumn{1}{c| }{0.00006} & \multicolumn{1}{c}{0.44343}\\ \hline \hline
\multirow{4}{3.65cm}{$(1.0,1.4,1.0,4.8,1.9,1.5)$} & (1,1) & 0.11511 & 0.11511 & 0.00065 & 0.97122\\
\cline{2-6} & \multicolumn{1}{c| }{(1,2)} & \multicolumn{1}{c| }{0.04233} & \multicolumn{1}{c| }{0.04237} & \multicolumn{1}{c| }{0.00034} & \multicolumn{1}{c}{0.55406}\\ 
\cline{2-6} & \multicolumn{1}{c| }{(2,1)} & \multicolumn{1}{c| }{0.05370} & \multicolumn{1}{c| }{0.05368} & \multicolumn{1}{c| }{0.00043} & \multicolumn{1}{c}{0.75197} \\ 
\cline{2-6} & \multicolumn{1}{c| }{(2,2)} & \multicolumn{1}{c| }{0.01718} & \multicolumn{1}{c| }{0.01718} & \multicolumn{1}{c| }{0.00018} & \multicolumn{1}{c}{0.93455}\\ \hline \hline
\multirow{4}{3.65cm}{$(1.7,3.1,2.4,2.9,3.7,0.8)$} & (1,1) & 0.11345 & 0.11347 & 0.00050 & 0.83070\\
\cline{2-6} & \multicolumn{1}{c| }{(1,2)} & \multicolumn{1}{c| }{0.05193} & \multicolumn{1}{c| }{0.05195} & \multicolumn{1}{c| }{0.00031} & \multicolumn{1}{c}{0.73831}\\ 
\cline{2-6} & \multicolumn{1}{c| }{(2,1)} & \multicolumn{1}{c| }{0.03747} & \multicolumn{1}{c| }{0.03748} & \multicolumn{1}{c| }{0.00028} & \multicolumn{1}{c}{0.91140} \\ 
\cline{2-6} & \multicolumn{1}{c| }{(2,2)} & \multicolumn{1}{c| }{0.01551} & \multicolumn{1}{c| }{0.01552} & \multicolumn{1}{c| }{0.00012} & \multicolumn{1}{c}{0.70376}
\end{tabular}
\small
\label{tab:moment.test.ncdir}
\end{table}

If on one hand the new representation of the Non-central Dirichlet model in Eq.~(\ref{eq:rappr.clc.ncdir}) only brings an element of theoretical elegance into the study of this latter, on the other the new formula for the mixed raw moments in Eq.~(\ref{eq:ncdir.momr1r2}) allows to overcome the disadvantages of the existing formula in Eq.~(\ref{eq:ncdir.mixr1r2mom.def2}) due to its doubly-infinite-series structure. As a matter of fact, the former formula is computationally less demanding than the latter thanks to its more easily handeable finite-sum form. Specifically, the derived moment formula requires a less implementation effort and calls for a less considerable execution time to produce the desired results than the existing one. In this regard, the computational performances of the two formulas are compared in the following way. The overall amount of machine-time spent to compute the mixed raw moments of order $(r_1,r_2) \in \left\{(1,1), (1,2), (2,1), (2,2) \right\}$ of the NcDir distribution by using both the derived and the existing formulas is measured $n=30$ times for each value of the parameter vector considered in Table~\ref{tab:moment.test.ncdir}. Then, the sample mean and standard deviation of this latter quantity are computed in each of the above cases. Hence, the null hypothesis that the true mean of the execution time of the new formula is not inferior to the one of the existing formula is checked by using the one-tailed $Z$ test for large samples. The achieved results are listed in Table~\ref{tab:moment.time.ncdir}, from which we conclude that the new formula is beyond doubt to be preferred to the existing one for its major computational efficiency. More precisely, the obtained findings suggest that on average the execution time of the existing formula is approximately of the order of more than 50 times slower than the one of the new formula.

\begin{table}[ht]
\caption{Means ($\bar x$) and standard deviations ($s$) of the overall machine-time (in seconds) to compute the mixed raw moments of order $(r_1,r_2) \in \left\{(1,1), (1,2), (2,1), (2,2) \right\}$ of the bivariate NcDir distribution by using the derived formula (``Sum'') and the existing formula (``Series'') for selected values of the parameter vector and $p$-values of the one-tailed $Z$ test to check the null hypothesis $H_0: \, \mu_{Sum}-\mu_{Series} \, \geq\, 0$ ($n=30$).}

\vspace{0.25cm}
\centering
\begin{tabular}{c||c|c|c|c}
\multirow{2}{3.25cm}{$\left(\alpha_1,\alpha_2,\alpha_3,\lambda_1,\lambda_2,\lambda_3\right)$} & \multicolumn{4}{c}{Time ($''$)} \\ 
\cline{2-5} & Formula & $\bar{x}$ & $s$ & $p$-value \\ \hline \hline
\multirow{2}{3.65cm}{$(0.5,0.6,0.4,1.7,6.4,3.8)$} & Sum & 0.01400 & 0.01429 & \multirow{2}{1.3cm}{$<.0001$} \\
\cline{2-4} & \multicolumn{1}{c| }{Series} & \multicolumn{1}{c| }{0.79700} & \multicolumn{1}{c| }{0.47738} & \multicolumn{1}{c}{} \\ \hline \hline
\multirow{2}{3.65cm}{$(0.2,0.3,1.6,1.3,5.5,4.2)$} & Sum & 0.00933 & 0.01143 & \multirow{2}{0.15cm}{$0$} \\
\cline{2-4} & \multicolumn{1}{c| }{Series} & \multicolumn{1}{c| }{0.61100} & \multicolumn{1}{c| }{0.02295} & \multicolumn{1}{c}{} \\ \hline \hline
\multirow{2}{3.65cm}{$(1.0,1.4,1.0,4.8,1.9,1.5)$} &  Sum & 0.01033 & 0.01217 & \multirow{2}{0.15cm}{$0$}   \\
\cline{2-4} & \multicolumn{1}{c| }{Series} & \multicolumn{1}{c| }{0.40400} & \multicolumn{1}{c| }{0.01354} & \multicolumn{1}{c}{} \\ 
\hline \hline
\multirow{2}{3.65cm}{$(1.7,3.1,2.4,2.9,3.7,0.8)$} & Sum & 0.00700 & 0.00952 & \multirow{2}{0.15cm}{$0$}   \\
\cline{2-4} & \multicolumn{1}{c| }{Series} & \multicolumn{1}{c| }{0.37500} & \multicolumn{1}{c| }{0.00900} & \multicolumn{1}{c}{} 
\end{tabular}
\small
\label{tab:moment.time.ncdir}
\end{table}

\noindent All the simulations and the procedures of interest have been carried out by means of routines implemented by the Author in the statistical environment \texttt{R}. 

\section{Conclusions}
\label{sec:concl}

In the present paper new light was shed on the analysis of the Non-central Dirichlet distribution. As a matter of fact, the remarkable mathematical complexity affecting the handling of such a probabilistic model was faced herein by introducing a novel approach to analyzing it based on a relationship of conditional independence and a simple conditional density. Resorting to this approach led us to achieve new results inherent in the stochastic representation and the computation of the mixed raw moments of the model under consideration. In particular, the derived moment formula, thanks to its finite-sum structure, guaranteed a greater computational efficiency than the existing one, this latter having an infinite-series structure. In conclusion, we hope this approach and the obtained results may attract wider applications of this model in statistics.    


\end{document}